\documentclass[12pt]{amsart} %
\usepackage{eurosym, color}

\usepackage{DKstyle}

\newcommand{\vertiii}[1]{{\left\vert\kern-0.25ex\left\vert\kern-0.25ex\left\vert #1
    \right\vert\kern-0.25ex\right\vert\kern-0.25ex\right\vert}}
\theoremstyle{plain}
\makeatletter
\newcommand*{\rom}[1]{\expandafter\@slowromancap\romannumeral #1@}
\makeatother

\newtheorem*{thm*}{Theorem}

\usepackage{enumitem}

\usepackage{caption}
\usepackage[labelfont=rm]{subcaption}
\usepackage{bm}
\usepackage[pagebackref=true, colorlinks]{hyperref}

\hypersetup{pdffitwindow=true,linkcolor=blue,citecolor=blue,urlcolor=blue,menucolor=blue}

\usepackage{comment}


\date{\today}%
\dedicatory{}%
\commby{}%

\title{Effective density of values of indefinite ternary  inhomogeneous quadratic forms}
\author{Dubi Kelmer}
\thanks{This work is partially supported by NSF CAREER grant DMS-1651563.}
\email{kelmer@bc.edu}
\address{Department of Mathematics, Boston College, Boston, Massachusetts, United States}
\date{\today}

\begin{document}

\maketitle
\begin{abstract}
\noindent    
Given an inhomogeneous quadratic form $Q_\xi(v)=Q(v+\xi)$ with $Q$ an indefinite $\Q$-isotropic rational ternary form and $\xi\in \R^3$ irrational, we  prove an effective lower bound for the number of integer vectors $v\in \Z^n$ with 
$\|v\|\leq T$ such that  $|Q_\xi(v)-t|<\delta$ that is valid for any $t\in \R$ and all $\delta\geq T^{-\nu}$, with $\nu>0$ depending explicitly on the Diophantine properties of $\xi$. In particular, for $\xi$ with algebraic entries we can take any $\nu<\frac{1}{8}$.
\end{abstract}

\section{Introduction}
Let $Q$ be an indefinite quadratic form in $n\geq 3$ variables, $\xi\in \R^n$, and $Q_\xi(v)=Q(v+\xi)$ the corresponding inhomogeneous form. We say that $Q_\xi$ is irrational if either $\xi \not\in \Q^n$ or $Q$ is not proportional to a irrational form.  Extending Margulis' proof of the Oppenheim conjecture, or more precisely the quantitative argument of \cite{EskinMargulisMozes1998}, it was shown in \cite{MargulisMohammadi11}, that for an  irrational inhomogeneous form the values at integer points, $Q_\xi(v)$ with $v\in \Z^n$, are dense in $\R$. More precisely, they showed that for irrational $Q_\xi$  and any fixed $\delta>0$, the counting functions  
\begin{equation}\label{e:NQ}
N_{Q,\xi,t}(T,\delta)=\{ v\in \Z^n: \|v\|\leq T, |Q_\xi(v)-t|\leq \delta\}
\end{equation} 
satisfy
$$\liminf_{T\to\infty} \frac{N_{Q,\xi,t}(T,\delta)}{\delta T^{n-2}}= C_{Q,\xi ,t}>0,$$ 
where the $\liminf$ can be replaced with a limit when the signature of $Q$ is not $(2,2)$ or $(2,1)$.

Given such a density result it is natural to ask about effective density, that is, how large do we need to take $v\in \Z^n$ so that $Q_\xi(v)$ is close to a given target $t\in \R$? To make this more precise we define the critical exponent $\omega_0=\omega_0(Q,\xi,t)$ as the supremum over all $\omega>0$ such that the inequalities 
$$\|v\|\leq T, \quad |Q_\xi(v)-t|\leq T^{-\omega}.$$
have integer solutions $v\in \Z^n$ for all sufficiently large $T$.

It was shown in \cite{GhoshKelmerYu2022}, that for any target $t\in\R$, for almost all inhomogeneous forms $Q_\xi$ one has $\omega_0(Q,\xi,t)\geq n-2$ (which is the best bound one could expect following the pigeonhole principle). Moreover, it was also shown that the same bound holds for any fixed rational $\xi\in \Q^n$ and for almost all indefinite forms $Q$. Using a different method, in \cite{GhoshKelmerYu2023} it was shown that for any fixed rational form $Q$ in $n\leq 4$ variables and for almost all $\xi\in \R^n$ one also has $\omega_0(Q,\xi, t)\geq n-2$ (for larger $n\geq 5$ a weaker lower bound was obtained). While the method of \cite{GhoshKelmerYu2022,GhoshKelmerYu2023} produce very good bounds generically, they say nothing about any specific form.

In this note we focus on the case when $Q$ is a rational indefinite $\Q$-isotropic ternary form and $\xi\in \R^3$ is irrationall.  
To simplify the discussion, we will work with the specific form 
\begin{equation}\label{e:Q}
Q(\alpha,\beta,\gamma)=\beta^2-4\alpha\gamma.
\end{equation}
Since any other rational indefinite $\Q$-isotropic ternary form, $Q'$, satisfies that $mQ'(v)=Q(vM)$ for some $m\in \Z$ and a non singular $M\in \mathrm{Mat}_3(\Z)$ we do not lose any generality by doing this.
We will give an effective bound for $\omega_0(Q,\xi,t)$, that holds for all $t$ as long as $\xi$ satisfies a Diophantine condition. More explicitly, we will give a lower bound for $N_{Q,\xi,t}(T,\delta)$ that is uniform in some range $\delta\geq T^{-\nu}$ with $\nu$ depending explicitly on Diophantine properties of $\xi$. To describe this condition we recall that  a real number $\alpha\in \R$ is called  $\kappa$-Diophantine if there is some $c>0$ such that $|q\alpha-p|\geq \frac{c}{q^\kappa}$ for all $q\in \N,\;p\in\Z$.

\begin{thm}\label{t:main}
Let $Q$ be as in \eqref{e:Q}. Let $\xi=(\alpha,\beta,\gamma)\in \R^3$ and assume that there are co-prime $a,c\in\Z$ such that $\alpha a^2+\beta ac+\gamma c^2$ is $\kappa$-Diophantine.
 Then, for any $\nu<\frac{1}{8\kappa}$ and for any $t\in \R$   there is some $c=c(\xi,\nu, t)>0$ and $T_0=T_0(\xi,\nu,t)$ such that for any $T\geq T_0$ and $\delta\geq T^{-\nu}$
$$N_{Q,\xi,t}(T, \delta)\geq  cT^{1/2}\delta^2.$$
\end{thm}
For any $\kappa>1$,  by Roth's Theorem \cite{Roth1955} any irrational algebraic $\alpha\in \R$ is $\kappa$-Diophantine, and by Khinchin's theorem this is also true for almost all $\alpha\in \R$.   We thus get the following.
\begin{cor}
For $\xi\in \R^3$ with algebraic coefficients, and also for almost every $\xi\in \R^3$, for any $t\in \R$ we have $\omega_0(Q,\xi,t)\geq \frac{1}{8}$.
\end{cor}
\begin{rem}
While the main consequence is an explicit bound for $\omega_0(Q,\xi,t)$ for a specific $\xi$,  even the result for almost all $\xi\in \R^3$ is new.  While \cite{GhoshKelmerYu2023} showed that for any fixed $t\in \R$ one has $\omega_0(Q,\xi,t)\geq 1$ for a full measure set of $\xi\in\R^3$, that set depends on the target $t$. So it is not clear if there is a full measure set of $\xi\in \R^3$ for which $\omega_0(Q,\xi,t)\geq 1$ for all $t\in \R$. However, the bound  $\omega_0(Q,\xi,t)\geq 1/8$ we obtain here for almost all $\xi$ can probably be improved by using appropriate adaptation of the method of  \cite{GhoshKelmerYu2023} using similar ideas to \cite{GhoshKelmer2018}.
\end{rem}

\begin{rem}
We expect the lower bound we have for $N_{Q,\xi,s}(T,\delta)$ is far from optimal.
Indeed, by \cite[Theorem 10]{MargulisMohammadi11}, for a fixed $\delta>0$, a target $t\in \R$  and Diophantine  $\xi\in \R^3$ we have that
$N_{Q,\xi,t}(T,\delta)\gg T\delta$. However, the implied constant here depends ineffectively on $\delta$, so can not be applied when $\delta$ shrinks with $T$. 
For the case of homogenous indefinite ternary forms, recent new results of \cite{LindenstraussMohammadiWangYang23} give effective lower bounds for $N_{Q,0,t}(T,\delta)$ when $Q$ is irrational with some Diophantine conditions on its coefficients, that is valid when $\delta\geq T^{-\nu}$ for some computable small $\nu>0$. It is possible that similar methods could produce the correct lower bound on $N_{Q,\xi,s}(T,\delta)$ also in this setting at least when $\delta$ is not too small with respect to $T$.
 \end{rem}

We end with a few words on the  proof. Let $G=\SO_Q(\R)$ and $\Gamma=\SO_Q(\Z)$. As was observed in \cite{GhoshKelmerYu2023} the problem of effective density of integer values of $Q_\xi$ can be reduced to an appropriate statement on effective density of 
$\Gamma$ orbits of $\xi$ on the torus $\bT^3=\R^3/\Z^3$.  Moreover, effective density results for $\Gamma$-orbits of a point $\xi\in \bT^3$ are equivalent to appropriate effective density results for the $G$-orbit of a points of the form $(g,\xi) \in G\ltimes \R^3/\G\ltimes \Z^3$. Now, such effective density results of $G$-orbits on  $G\ltimes \R^3/\G\ltimes \Z^3$ would follow from an appropriate effective equidistribution result for long unipotent orbits and it is possible that the methods developed in \cite{LindenstraussMohammadiWang22} could be used to produce such results. Alternatively, noting that $\SO_Q(\R)\cong \SL_2(\R)$ the method developed in \cite{StrombergssonVishe2020} to establish effective equidistribution for certain unipotent orbits in  $\SL_2(\R)\ltimes (\R^2)^{\oplus k}/\SL_2(\Z)\ltimes (\Z^2)^{\oplus k}$ using Fourier analysis (as well as more recent work \cite{SodergrenStrombergssonVishe2024} also using the circle method to handle more general orbits), could also potentially be adapted to this setting to give effective density. Indeed, this was done for forms of signature $(2,2)$, though, extending their approach to apply to our setting seems much harder. Moreover, while such methods could potentially give an effectively computable lower bound for $\omega_0(Q,\xi,t)$, such bounds would be very small.

The approach we take is quite different. Instead of working with homogenous dynamics we try to prove an effective density result for the action of $\Gamma$ on the torus directly.
While analyzing the full $\Gamma$ orbit in $\bT^3$ seems complicated, it turns out that to show density of integer values of $Q_\xi$ it is enough to consider a one parameter discrete unipotent subgroup inside $\Gamma$ (which exists precisely when $Q$ is isotropic). This action fixes one of the coordinates and gives an action on $\bT^2$. We can thus reduce the problem to proving effective density of a certain (quadratic) polynomial map on the torus. Such effective density results then follow from an effective version of Weyl's equidistribution theorem.

\section{Effective density  on the torus}
Fix $\alpha,\beta,\gamma\in \R$ and consider the map $\phi: \N\to \bT^2$ given by
\begin{equation}\label{e:phi}
\phi(m)=(2\alpha m+\beta,\alpha m^2+\beta m+\gamma) \pmod{1}
\end{equation}
Weyl's equidistribution theorem implies that as long as $\alpha$ is irrational the points $\{\phi(m)\}_{m\in \N}$ become equidistributed on $\bT^2$. Following the same argument we now prove an effective density result depending explicitly on the Diophantine properties of $\alpha$.  

We first introduce some notation. We denote by $\|\cdot\|$ the Euclidean norm  in $\R^n$ and by $\|\cdot\|_{\bT^n}$ the corresponding norm on the torus, that is 
$\|x\|_{\bT^n}=\min\{ \|x+v\|: v\in\Z^n\}$. For doing harmonic analysis on the torus we will use the notation $e(x)=e^{2\pi ix}$. 
We will use the notation $X\ll Y$ as well as  $X=O(Y)$ to mean that there is a constant $c>0$ such that $X\leq cY$ and we denote $X\asymp Y$ if $ X\ll Y\ll X$.
If we wish to emphasize the dependence of the implied constant on additional parameters we will indicate this as a subscript. 

The main result of this section is the following.
\begin{prop}\label{p:TorusDensity}
Given $(\alpha,\beta,\gamma)\in \R^3$ let $\phi:\N\to\bT^2$ be as in \eqref{e:phi}.
For any $v_0\in \bT^2,\; T\geq 1$ and small $\delta>0$ consider the counting function
 $$N_{\phi}(T,\delta)=\#\{1\leq m\leq T: \|\phi(m)-v_0\|_{\bT^2}\leq \delta\}.$$
Assume that $\alpha$ is $\kappa$-Diophantine and let $\nu<\frac{1}{4\kappa}$. Then there is some $T_0\geq 1$ (depending on $\nu$ and $\kappa$) such that for all $T\geq T_0$ and any  $T^{-\nu}\leq \delta<1$,  for any $v_0\in \bT^2$
$$N_{\phi}(T,\delta)\asymp T\delta^2.$$
\end{prop}

For the proof we will need the following two standard estimates. While these results are not new and the arguments are standard we will include a short proof for the sake of completeness.
The first estimate is Weyl's differencing trick.
\begin{lem}\label{l:ST}
For $\alpha,\beta \in \R$ and $n\in \N$  let 
$$S_T(n,\alpha,\beta)=\sum_{m=1}^T e(n\alpha m^2+\beta m).$$
Then 
$$|S_T(n,\alpha,\beta)|^2\leq T+2\sum_{m=1}^T\min\{ \frac{1}{\|2nm\alpha\|_\bT},T\}.$$
\end{lem}
\begin{proof}
Squaring $S_T(n,\alpha,\beta)$ the diagonal terms give $T$ and we can bound the non diagonal terms as follows.
\begin{eqnarray*}
|S_T(n,\alpha,\beta)|^2 &= & T+2\Re(\sum_{ m_1< m_2\leq T}e(n\alpha (m_1^2-m_2^2)+\beta (m_1-m_2))|)\\
&\leq  & T+2|\sum_{1\leq m_1\leq T}\sum_{h=1}^{T-m_1}e(-n\alpha 2m_1h -nh^2 -\beta h)|\\
&\leq & T+2\sum_{h=1}^{T}|\sum_{m_1=1}^{T-h}e(-2nh\alpha m_1)|\\
&\leq & T+2\sum_{h=1}^T \min\{\frac{1}{\|2nh\alpha\|_\bT},T)
\end{eqnarray*}
\end{proof}
The second estimate  is the following.
\begin{lem}\label{l:sumfrac}
Let $\alpha\in \R$  be $\kappa$-Diophantine for some $\kappa\geq 1$.  Then for any $T\geq 1$ and $M\geq 1$
$$\sum_{m=1}^{MT} \min\{\frac{1}{\|m\alpha\|_{\bT}}, T) \ll  MT^{2-1/\kappa}+MT\log(T).$$
\end{lem}
\begin{proof}
By Dirichlet approximation theorem for any $T\geq 1$  there are $p,q\in \Z$ with $1\leq q\leq T$ such that $|\alpha-\frac{p}{q}|\leq \frac{1}{Tq}$.
Then for any $i\neq j$ we have that $\|(i-j)\alpha\|_{\bT}\geq \|(i-j)\frac{p}{q}\|_{\bT}-\frac{|i-j|}{qT}$. Now if we further assume $|i-j|\leq \frac{q}{2}$ then  $ip\not\equiv jp\pmod{q}$ so $\|(i-j)\frac{p}{q}\|_{\bT}\geq \frac{1}{q}$ and $\frac{|i-j|}{qT}\leq \frac{1}{2q}$, hence $\|(i-j)\alpha\|_{\bT}\geq \frac{1}{2q}$. 

Now, we split the interval $1\leq m\leq MT$ into $[\frac{2MT}{q}]+1$ intervals each of lengths $\leq q/2$.
Let $I$ be one such interval. For each $m\in I$ we may pick a representative of $m\alpha\pmod{1}$ in $[-\frac{1}{2},\frac{1}{2}]$ so that $\|m\alpha\|_{\bT}$ is just the absolute value of that representative. Let $0<\alpha_1<\alpha_2<\ldots<\alpha_n<\frac{1}{2}$ denote the positive representatives (we can deal with negative representatives similarly).  Recall that for any $1\leq i<j\leq n$ we have $\alpha_j-\alpha_i\geq \frac{1}{2q}$, so $\alpha_j\geq \frac{j-1}{2q}$ for $j\geq 2$. We can thus bound, recalling that $n\leq \frac{q}{2}\leq  T$,
$$\sum_{j=1}^n \min\{T,\frac{1}{\alpha_j}\}\leq  T+\sum_{j=2}^n \frac{2q}{j-1} \leq T+2q\log(T).$$
We can similarly bound the contribution of the negative representatives in the same way, and adding up the contribution of the  $[\frac{2MT}{q}]+1$ intervals we get that
$$\sum_{m=1}^{MT} \min\{\frac{1}{\|m\alpha\|_{\bT}}, T) \leq \frac{4MT^2}{q}+8(M+1)T\log(T).$$
Finally, the Diophantine condition on $\alpha$ and the estimate $|\alpha-\frac{p}{q}|\leq \frac{1}{Tq}$ implies that  there is $C>0$ such that $q\geq CT^{1/\kappa}$ concluding the proof.

\end{proof}

We can now give the following
\begin{proof}[Proof of \propref{p:TorusDensity}]
Let $f\in C^\infty([-1/2,1/2])$ be nonnegative  bounded, having mean $1$ and vanishing outside of $[-1/4,1/4]$.
 Extend $f$  periodically to a function on $\R/\Z$ (that by a slight abuse of notation we still denote by $f$) and let $F(x,y)=f(x)f(y)$.  For $v_0\in \bT^2$ and $\delta>0$ let 
$F_\delta(v)=F(\frac{v-v_0}{\delta})$ so that 
$$N_{\phi}(T,\delta)\asymp \sum_{m=1}^T F_\delta(\phi(m)).$$
Now expand $F_\delta$ to its Fourier expansion noting that $\hat{F}_\delta(n)=\delta^2\hat{F}(\delta n)e(v_0\cdot n)$ for $n\in \Z^2$. Hence
\begin{eqnarray*}
N_{\phi}(T,\delta)&\asymp& T\delta^2(1+\frac{1}{T}\sum_{n\neq (0,0)}e(v_0\cdot n) \hat{F}(\delta n)\sum_{m=1}^T e(n_2\alpha m^2+(n_2 \beta+2n_1\alpha)m +n_2\gamma+n_1\beta)\\
&=& T\delta^2(1+\cE_1(T,\delta)+\cE_2(T,\delta))
\end{eqnarray*}
where $\cE_1$ is contribution of terms with $n_2=0$ and $\cE_2$ accounts for all other terms.
We first bound $\cE_1$. Collecting together terms with $\ell \leq |\delta n_1|\leq \ell+1$ and using the fast decay of Fourier coefficients, $|\hat{f}(t)| \ll  \frac{1}{t^{4}}$ for $|t|\geq 1$, we can bound  
\begin{eqnarray*}
|\cE_1(T,\delta)|&\leq & \frac{1}{T}\sum_{n_1\neq 0}|\hat{f}(\delta n_1)||\sum_{m=1}^T e(2n_1\alpha m+n_1\beta)|\\
&\ll & \frac{1}{T}\sum_{\ell=1}^\infty \frac{1}{\ell^4}\sum_{\ell-1\leq |n_1\delta|\leq \ell}|\sum_{m=1}^T e(2n_1\alpha m)|\\
&\ll& \frac{1}{T}\sum_{\ell=1}^\infty \frac{1}{\ell^4}\sum_{ |n_1|\leq \ell/\delta }\min\{\frac{1}{\|2n_1\alpha\|_{\bT}},T) 
\end{eqnarray*}
We can use \lemref{l:sumfrac} with $M=\frac{\ell}{\delta T}$ to bound the last inner sum by $O(\frac{\ell}{\delta} T^{1-1/\kappa}+\frac{\ell}{\delta}\log(T))$ and since the sum over $\ell$ absolutely converges  we can bound 
$$|\cE_1(T,\delta)|\ll \frac{1}{\delta T^{1/\kappa}}+\frac{\log(T)}{\delta T}\leq \frac{\log(T)}{T^{3/4\kappa}}$$  
where we used that $\delta>T^{-\nu}$ with $\nu<\frac{1}{4\kappa}$.

Next we bound the second term by
\begin{eqnarray*}
|\cE_2(T,\delta)|&\leq & \frac{1}{T}\sum_{n_1}|\hat{f}(\delta n_1)|\sum_{n_2\neq 0} |\hat{f}(\delta n_2)||\sum_{m=1}^T e(n_2\alpha m^2+(n_2 \beta+2n_1\alpha)m)|\\
&=&\frac{1}{T}\sum_{n_1}|\hat{f}(\delta n_1)|\sum_{n_2\neq 0} |\hat{f}(\delta n_2)|S_T(n_2,\alpha,\tilde\beta)|\\
\end{eqnarray*}
where $\tilde\beta=n_2\beta+2n_1\alpha$. For the sum over $n_2$ again collect together terms with $\ell\leq |\delta n_2|\leq \ell+1$, to bound 
$$|\sum_{n_2\neq 0} |\hat{f}(\delta n_2)|S_T(n_2,\alpha,\tilde\beta)|\ll\sum_{\ell=1}^\infty  \ell^{-4}\sum_{ n_2=1}^{\ell/\delta}|S_T(n_2,\alpha,\tilde\beta)|.$$
Using Cauchy-Schwarz we can bound the most inner sum  by
$$ \sum_{n_2\leq \frac{\ell}{\delta}}|S_T(n_2,\alpha,\tilde\beta)|\leq (\ell/\delta)^{1/2}(\sum_{n\leq \ell/\delta}|S_T(n_2,\alpha,\tilde\beta)|^2)^{1/2},$$
and by \lemref{l:ST} 
\begin{eqnarray*}
\sum_{n_2\leq \ell/\delta}|S_T(n_2,\alpha,\tilde\beta)|^2&\ll& \sum_{n_2\leq \ell/\delta}(T+\sum_{m=1}^T \min\{\frac{1}{\|2n_2m\alpha\|_{\bT}},T\}).
\end{eqnarray*}
The first term is $O(\frac{\ell T}{\delta})$.  For the second term, writing $a=n_2m$ and $\sigma(a)=\sum_{d|a}1$, noting that $\sigma(a)\ll_\epsilon a^\epsilon$ we can bound
\begin{eqnarray*}\sum_{n_2\leq \ell/\delta}\sum_{m=1}^T \min\{\frac{1}{\|2n_2m\alpha\|_{\bT}},T)&\leq &\sum_{a\leq \frac{2\ell T}{\delta}} \sigma(a)\min\{\frac{1}{\|2a\alpha\|_{\bT}},T\}\\
&\ll_\epsilon & ( \frac{\ell T}{\delta})^\epsilon\sum_{a\leq \frac{2\ell T}{\delta}}\min\{\frac{1}{\|2a\alpha\|_{\bT}},T\}\\
&\ll_\epsilon & (\frac{\ell}{\delta})^{1+\epsilon} T^{2-1/\kappa+\epsilon}
\end{eqnarray*}
where we used \lemref{l:sumfrac} with $M=\frac{2\ell}{\delta}$  to bound the last sum.

We thus get that
$$\sum_{ n_2=1}^{\ell/\delta}|S_T(n_2,\alpha,\tilde\beta)|\ll_\epsilon (\frac{\ell}{\delta})^{1+\epsilon} T^{1-1/2\kappa+\epsilon},$$
and plugging this back, noting that the sum over $\ell$ absolutely converges we get that 
$$\sum_{n_2\neq 0} |\hat{f}(\delta n_2)|S_T(n_2,\alpha,\tilde\beta)|\ll  (\frac{1}{\delta})^{1+\epsilon} T^{1-1/2\kappa+\epsilon}.$$
Next for the sum over $n_1\in \Z$ we note that $\sum_{n_1}|\hat{f}(\delta n_1)|\ll \frac{1}{\delta}$ to get that for $\delta\geq T^{-\nu}$
$$|\cE_2(T,\delta)|\ll (\frac{1}{\delta T^{1/4\kappa}})^{2+\epsilon}\leq T^{-(\frac{1}{4\kappa}-\nu)(2+\epsilon)}\ .$$
Since we assume $\nu<\frac{1}{4\kappa}$ we have that $\cE_2(T,\delta)\to 0$ as $T\to \infty$ and hence for all  $T$ sufficiently large we can bound $|\cE_1(T,\delta)+\cE_2(T,\delta)|<\frac{1}{2}$ and conclude that 
$N_{\phi}(T,\delta)\asymp T\delta^2$ as claimed. 
\end{proof}

\section{Proof of \thmref{t:main}}
For  $Q$ as in \eqref{e:Q} let $\iota: \SL_2(\R)\to\SO_Q(\R)$
denote the standard homomorphism given by 
$$\iota(\begin{pmatrix} a & b\\c &d\end{pmatrix}=\begin{pmatrix}a^2 & 2ab & b^2\\ ac & ad+bc& bd\\ c^2 & 2cd & d^2\end{pmatrix},$$
acting from the right on row vectors $(\alpha,\beta,\gamma)\in \R^3$.
For any $m\in \N$ let  
$$M_m=\iota(\begin{pmatrix} 1 & m\\ 0 &1\end{pmatrix})= \begin{pmatrix} 1 & 2m & m^2\\ 0 & 1& m\\ 0 & 0 & 1\end{pmatrix}\in \SO_Q(\Z),$$
and note that
$$(\alpha,\beta,\gamma )M_m= (\alpha, 2\alpha m+\beta, \alpha m^2+\beta m+\gamma)=(\alpha,\phi(m)),$$
with $\phi$ given in \eqref{e:phi}.
We can now give the following.
\begin{proof}[Proof of Theorem \ref{t:main}]
If needed we may replace  $\xi=(\alpha,\beta,\gamma)$ with  $\tilde{\xi}=\xi M=(\tilde\alpha,\tilde\beta,\tilde\gamma)$ for some fixed $M\in \SO_Q(\Z)$ so that $\tilde{\alpha}=a^2\alpha+ac\beta+c^2\gamma $ is $\kappa$-Diophantine. Since $Q_\xi(v)=Q_{\tilde \xi}(v M)$ we may without loss of generality assume that $\alpha$ is $\kappa$-Diophantine. 

Now, given a target $t\in \R$ we can find $(y,z)\in \R^2$ such that $Q(\alpha,y,z)=y^2-4\alpha z=t$ and let $\eta=(\alpha,y,z)$.  Let $\nu=2\omega<\frac{1}{4\kappa}$, then by \propref{p:TorusDensity}  for all sufficiently large $T$ and for any $\delta>(\sqrt{T})^{-\nu}=T^{-\omega}$
 $$\#\{1\leq m\leq \sqrt{T}: \|\phi(m)-(y,z)\|_{\bT^2}\leq \delta \}\asymp \sqrt{T} \delta^2.$$
For any $m$ in this set, there is some $u_m=(0,a_m,b_m)\in \Z^3$ such that  
$$\|\xi M_m+u_m-\eta\|\leq \delta,$$
and we denote by $v_m=u_mM_m^{-1}$. Noting that $Q(vM_m)=Q(v)$ we can estimate 
$$|Q_\xi(v_m)-s|=|Q(\xi+v_m)-Q(\eta)|=|Q(\xi M_m+u_m)-Q(\eta)|\ll_{y,z} \delta.$$ 

We can now bound the size of $v_m$ as follows. Noting that $\|\xi M_m+u_m-\eta\|\leq \delta$ and that $\|M_m\|\ll m^2\leq T$ we see that 
$\|\xi +v_m-\eta M_m^{-1}\| \ll \delta T$ and hence 
$$\|v_m\|\ll \|\eta M_m^{-1}\|+\delta T\ll T.$$

Finally note that if $m\neq m'$ then $u_m\neq u_{m'}$ (since otherwise  $\|\xi M_m-\xi M_{m'}\|\leq 2\delta$ which for $\delta$ sufficiently small  can only happen when $m=m'$). Since $u_m=(0,a_m,b_m)$ then $v_m=u_m M_m^{-1}=(0,a_m,b_m-ma_m)$, and the condition $u_m\neq u_{m'}$ implies that $v_m\neq v_{m'}$. 

Now, replacing $\delta$ and $T$ by a constant multiple  we see that there is some $c>0$ so that the set
$$\{v_m: 1\leq m\leq c\sqrt{T}: \|\phi(m)-(y,z)\|_{\bT^2}\leq c\delta\}$$ 
is a set containing $\asymp T\delta^2$ distinct vectors of size $\|v_m\|\leq T$ with $|Q_\xi(v)-s|\leq \delta,$
thus concluding the proof.
\end{proof}

%
%

\end{document}